\documentclass[11pt]{amsart}

\usepackage{amsfonts, amssymb, amscd}
\numberwithin{equation}{section}

\usepackage[symbol]{footmisc}

\usepackage{todonotes}

\usepackage{bm}
\usepackage{verbatim}
\usepackage{amssymb}
\usepackage{mathrsfs}
\usepackage{graphicx}
\usepackage{tikz-cd}
\usepackage{subcaption}
\usepackage{listings}
\usepackage{subfiles}
\usepackage[toc,page]{appendix}
\usepackage{mathtools}
\usepackage{comment}
\usepackage{enumerate}
\usepackage{enumitem}
\usepackage[linesnumbered,ruled]{algorithm2e}

\usepackage{graphicx}
\graphicspath{{images/}}

\usepackage{appendix}
\usepackage{hyperref}
\newcommand{\Qq}{\mathbb{Q}}

\newcommand{\Rr}{\mathbb{R}}

\newcommand{\Zz}{\mathbb{Z}}

\newcommand{\vol}{\operatorname{vol}}

\newcommand{\lcm}{\operatorname{lcm}}

\newcommand{\Supp}{\operatorname{Supp}}

\newcommand{\mult}{\operatorname{mult}}

\newcommand{\Ii}{{\Gamma}}

\newtheorem{thm}{Theorem}[section]

\newtheorem{cor}[thm]{Corollary}
\newtheorem{lem}[thm]{Lemma}

\newtheorem{claim}[thm]{Claim}
\theoremstyle{definition}
\newtheorem{rem}[thm]{Remark}

\newtheorem{ex}[thm]{Example}

\theoremstyle{definition}
\newtheorem{defn}[thm]{Definition}

\begin{document}

\title{Effective birationality for sub-pairs with real coefficients}

\author{Jingjun Han and Jihao Liu}

\begin{abstract}
For $\epsilon$-lc Fano type varieties $X$ of dimension $d$ and a given finite set $\Ii$, we show that there exists a positive integer $m_0$ which only depends on $\epsilon,d$ and $\Ii$, such that both $|-mK_X-\sum_i\lceil mb_i\rceil B_i|$ and $|-mK_X-\sum_i\lfloor mb_i\rfloor B_i|$ define birational maps for any $m\ge m_0$ provided that $B_i$ are pseudo-effective Weil divisors, $b_i\in\Ii$, and $-(K_X+\sum_ib_iB_i)$ is big. When $\Ii\subset[0,1]$ satisfies the DCC but is not finite, we construct an example to show that the effective birationality may fail even if $X$ is fixed, $B_i$ are fixed prime divisors, and $(X,B)$ is $\epsilon'$-lc for some $\epsilon'>0$.
\end{abstract}

\address{Department of Mathematics, Johns Hopkins University, Baltimore, MD 21218, USA}
\email{jhan@math.jhu.edu}

\address{Department of Mathematics, The University of Uath, Salt Lake City, UT 84112, USA}
\email{jliu@math.utah.edu}

\date{\today}

\maketitle

\tableofcontents

\section{Introduction}

We work over the field of complex numbers $\mathbb C$.

For any Fano variety $X$, the anti-pluricanonical systems $|-mK_X|$ determine its geometry to a large
extent. In the past decade, one of the most important progress in birational geometry  is the effective birationality for $\epsilon$-lc Fano varieties proved by Birkar, which leads to a solution of the Borisov-Alexeev-Borisov (BAB) conjecture:

\begin{thm}\cite[Theorem 1.2]{Bir19}\label{thm: birkar -mK_X}
Let $d$ be a positive integer and $\epsilon$ a positive real number. Then there exists a positive integer $m$ depending only on $d$ and $\epsilon$ such that if $X$ is any $\epsilon$-lc weak Fano variety of dimension $d$, then $|-mK_{X}|$ defines a birational map.
\end{thm}

It is natural to ask whether Theorem \ref{thm: birkar -mK_X} can be improved to the case of log pairs $(X,B)$. This is inspired by similar works in the general type case. The effective birationality of the linear systems $|mK_X|$ for smooth varieties $X$ of general type was proved by Hacon-M\textsuperscript{c}Kernan \cite{HM06}, and Takayama \cite{Tak06} independently. Later, in a series of celebrated papers, Hacon, M\textsuperscript{c}Kernan and Xu showed that the effective birationality of the linear systems $|m(K_X+B)|$ for lc pairs $(X,B)$ of log general type, provided that the coefficients of $B$ satisfy the descending chain condition (DCC) \cite{HMX13,HMX14}. Here we adopt the convention that $|D|:=|\lfloor D\rfloor|$ for any $\Rr$-divisor $D$.  Unfortunately, Example \ref{ex: gen hyp p11910} indicates that we cannot generalize Birkar's result straightforwardly, even when $\dim X=2$ and the coefficients set $\Ii$ satisfies the DCC. Nevertheless, when $\Ii$ is a finite set, the effective birationality holds even when we loose our assumption to a wider class of sub-pairs $(X,B)$.

\begin{thm}\label{thm: eff bir fin coeff gpair}
Let $d$ be a positive integer, $\epsilon$ a positive real number, and $\Ii_0$ a finite set of non-negative real numbers. Then there exists a positive integer $m_0$ depending only on $d,\epsilon$ and $\Ii_0$ satisfying the following. Assume that
	\begin{enumerate}
 \item $(X,\Delta)$ is projective $\epsilon$-lc of dimension $d$ for some boundary $\Delta$ such that $K_X+\Delta\sim_{\mathbb R}0$ and $\Delta$ is big,
 \item $B=\sum_{i=1}^s b_iB_i$ is a pseudo-effective $\mathbb R$-divisor on $X$, where each $b_i\in\Ii_0$, and each $B_i$ is a pseudo-effective Weil divisor, and
 \item $-(K_X+B)$ is big,
	\end{enumerate}
	then $|-mK_X-\sum_{i=1}^s\lceil mb_i\rceil B_i|$ and $|-mK_X-\sum_{i=1}^s\lfloor mb_i\rfloor B_i|$ define birational maps for any integer $m\geq m_0$ respectively. In particular, if each $B_i\geq 0$, then $|-m(K_X+B)|$ defines a birational map for any integer $m\geq m_0$.
\end{thm} 
In Theorem \ref{thm: eff bir fin coeff gpair}, $(X,B)$ may not even be sub-lc as $B_i$ may not be effective and $\Ii_0$ may not belong to $[0,1]$. It is also possible that $|-mK_X-\sum_{i=1}^s\lceil mb_i\rceil B_i|\not\subset |-mK_X-\sum_{i=1}^s\lfloor mb_i\rfloor B_i|$ although $\sum_{i=1}^s(\lceil mb_i\rceil-\lfloor mb_i\rfloor)B_i$ is pseudo-effective. 

\medskip

When $\vol(-(K_X+B))$ is bounded from below away from $0$, we can even remove the assumption on the coefficients $b_i$ in Theorem \ref{thm: eff bir fin coeff gpair}.

\begin{thm}\label{thm: eff bir bound volume more section gpair}
Let $d$ be a positive integer, and $\epsilon,v$ and $\delta$ three positive real numbers. Then there exist a positive integer $m_0$ depending only on $d,\epsilon$ and $v$, and a positive integer $m_0'$ depending only on $d,\epsilon,v$ and $\delta$ satisfying the following. Assume that
	\begin{itemize}
 \item $(X,\Delta)$ is projective $\epsilon$-lc of dimension $d$ for some boundary $\Delta$ such that $K_X+\Delta\sim_{\mathbb R}0$ and $\Delta$ is big,
 \item $B$ is a pseudo-effective $\mathbb R$-divisor on $X$, and
 \item $\vol(-(K_X+B))\geq v$,
	\end{itemize}
		then  
	\begin{enumerate}
	\item $|\lceil- m(K_X+B)\rceil|$ defines a birational map for any integer $m\geq m_0$,
	\item if $B=\sum_{i=1}^s b_iB_i$, where each $B_i$ is a pseudo-effective Weil divisor, then
	\begin{enumerate}
	    \item $|-mK_X-\sum_{i=1}^s\lfloor mb_i\rfloor B_i|$ defines a birational map for any integer $m\geq m_0$, and
	    \item  if $b_i\geq\delta$ for every $i$,  then  $|-mK_X-\sum_{i=1}^s\lceil mb_i\rceil B_i|$ defines a birational map for any integer $m\geq m_0'$. In particular, if each $B_i\geq 0$, then $|-m(K_X+B)|$ defines a birational map for any integer $m\geq m_0'$.
	\end{enumerate}
	\end{enumerate}
\end{thm}

Indeed, Theorem \ref{thm: eff bir fin coeff gpair} follows from Theorem \ref{thm: eff bir bound volume more section gpair} and a detailed study on the volume $\vol(-(K_X+B))$:
\begin{thm}\label{thm: bddvolacc maynotbelc}
	Let $d$ be an integer, $\epsilon$ a positive real number, and $\Ii$ a DCC (resp. finite) set of non-negative real numbers. Then there exists an ACC (resp. finite) set $\Ii'$ depending only on $d,\epsilon$ and $\Ii$ satisfying the following. Assume that
	\begin{enumerate}
	    \item $(X,\Delta)$ is projective $\epsilon$-lc of dimension $d$ for some boundary $\Delta$ such that $K_X+\Delta\sim_{\mathbb R}0$ and $\Delta$ is big,
	    \item $B$ is a pseudo-effective $\Rr$-divisor on $X$, and
	    \item $B=\sum_{i=1}^s b_iB_i$, where each $b_i\in\Ii$, and each $B_i$ is a pseudo-effective Weil divisor.
	\end{enumerate}
Then $\vol(-(K_X+B))\in \Ii'$. 
\end{thm}

\begin{rem}
A special case when we can apply Theorem \ref{thm: eff bir fin coeff gpair}, Theorem \ref{thm: eff bir bound volume more section gpair}, and Theorem \ref{thm: bddvolacc maynotbelc} is when $(X,B:=B'+M)$ is a generalized polarized pair, where $M$ is a nef $\Qq$-Cartier combinations (NQC) $\Rr$-divisor. We refer the readers to \cite{BZ16,HL18,LT19,HL19,HL20} in this direction.

\begin{rem}
The assumptions of Theorem \ref{thm: eff bir fin coeff gpair}, Theorem \ref{thm: eff bir bound volume more section gpair}, and Theorem \ref{thm: bddvolacc maynotbelc} are necessary: see Section 6 for some examples.
\end{rem}


\end{rem}

We remark that Birkar proved Theorem \ref{thm: birkar -mK_X} together with Shokurov's conjecture on boundednss of lc complements when $\Ii\subset[0,1]\cap \Qq$ is finite \cite[Theorem 1.1, Theorem 1.7, Theorem 1.8]{Bir19}. Shokurov's conjecture on boundednss of lc complements (when $\Ii$ satisfies the DCC) was proved by the authors and Shokurov \cite[Theorem 1.10]{HLS19}, and Chen  generalized it to generalized polarized pairs \cite[Theorem 1.1]{Che20}. The proofs of Theorem \ref{thm: eff bir fin coeff gpair} and Theorem \ref{thm: eff bir bound volume more section gpair} do not depend on those results in \cite{HLS19,Che20}.

\medskip

We also point out that in Theorem \ref{thm: eff bir fin coeff gpair} and Theorem \ref{thm: eff bir bound volume more section gpair}, we show the effective birationality for any positive integer $m\ge m_0$ rather than for a fixed positive integer $m$ (c.f. Theorem \ref{thm: birkar -mK_X}). In this paper, we will also generalize other previous results on effective birationality, \cite[Theorem 1.3(3)]{HMX14}, \cite[Theorem 1.3]{BZ16}, to this setting. See Theorem \ref{thm: birational one m to inf m} and Theorem \ref{thm: IMPROVED all big theorems on effective birationality}. We also slightly generalize  \cite[Theorem 1.2]{BZ16}, see Theorem \ref{thm: effective iitaka sufficiently large}. 

\medskip


\noindent\textit{Structure of the paper}. In Section 2, we introduce some notation and tools that will be used in the rest of the paper. In Section 3, we study the volume $\vol(-(K_X+B))$ and prove Theorem \ref{thm: bddvolacc maynotbelc}. In Section 4 we prove Theorem \ref{thm: eff bir fin coeff gpair} and Theorem \ref{thm: eff bir bound volume more section gpair}. In Section 5, we adopt the methods developed in this paper to slightly generalize \cite[Theorem 1.3(3)]{HMX14} and \cite[Theorem 1.2, Theorem 1.3]{BZ16}. In Section 6, we give some examples to illustrate that the assumptions in the theorems we proved are necessary.





\medskip

\noindent\textbf{Acknowledgement}. The authors would like to thank Guodu Chen, Chen Jiang and Yuchen Liu for useful comments and discussions. The first author would like to thank Chenyang Xu for encouraging him to start this work. The second author would like to thank Christopher D. Hacon for useful discussions and encouragements. The second author was partially supported by NSF research grants no: DMS-1801851, DMS-1952522 and by a grant from the Simons Foundation; Award Number: 256202.

\section{Preliminaries}

\subsection{Pairs and singularities}
\begin{defn}\label{defn: DCC and ACC}
	Let $\Ii$ be a set of real numbers. We say that $\Ii$ satisfies the \emph{descending chain condition} (DCC) if any decreasing sequence $a_1\ge a_2 \ge \cdots \ge a_k \ge\cdots$ in $\Ii$ stabilizes. We say that $\Ii$ satisfies the \emph{ascending chain condition} (ACC) if any increasing sequence in $\Ii$ stabilizes. 
	\end{defn}
	
	\begin{defn}
	Let $X$ be a normal variety, and $B=\sum_{i=1}^s b_iB_i$ an $\Rr$-divisor on $X$, where $B_i$ are the irreducible components of $B$. $b_i$ are called the \emph{coefficents} of $B$, $||B||:=\max_{1\leq i\leq s}\{|b_i|\}$, and we write $B\in\Ii$ if $b_i\in\Ii$ for every $i$. We write $B\geq 0$ if $B\in [0,+\infty)$. If $B\sim_{\mathbb R}B'\geq 0$ for some $B'$, $B$ is called \emph{effective}.
\end{defn}

\begin{defn}\label{defn: positivity}
	A \emph{sub-pair} $(X,B)$ consists of a normal quasi-projective variety $X$ and an $\Rr$-divisor $B$ such that $K_X+B$ is $\Rr$-Cartier. If $B\geq 0$, then $(X,B)$ is called a \emph{pair}. If $B\in [0,1]$, then $B$ is called a \emph{boundary} of $X$.
	
	Let $E$ be a prime divisor on $X$ and $D$ an $\mathbb R$-divisor on $X$. 	We define $\mult_ED$ to be the \emph{multiplicity} of $E$ along $D$. 
	
	Let $\phi:W\to X$
	be any log resolution of $(X,B)$, and let
	$$K_W+B_W:=\phi^{*}(K_X+B).$$
	The \emph{log discrepancy} of a prime divisor $D$ on $W$ with respect to $(X,B)$ is $1-\mult_{D}B_W$ and it is denoted by $a(D,X,B).$
	For any positive real number $\epsilon$, we say that $(X,B)$ is \emph{klt} (resp. \emph{lc}, \emph{$\epsilon$-lc}) if $a(D,X,B)>0$, (resp. $\geq0$, $\geq\epsilon$) for every log resolution $\phi: W\to X$ as above and every prime divisor $D$ on $W$.  We say that $(X,B)$ is \emph{dlt} if $a(D,X,B)>0$ for any exceptional prime divisor $D\subset W$ over $X$ for some log resolution $\phi:W\rightarrow X$.
	\end{defn}

\begin{defn}
Let $X$ be a normal projective variety. We say $X$ is of \emph{Fano type} if $(X,B)$ is klt and $-(K_X+B)$ is big and nef for some boundary $B$. In particular, $-K_X$ is big.
\end{defn}

\begin{defn}
Let $X$ be a normal projective variety. A \emph{big} $\Rr$-divisor on $X$ is an $\Rr$-divisor $D$ on $X$ such that $D\sim_{\mathbb R}A+E$, where $A$ is an ample $\Rr$-divisor and $E\geq 0$. We emphasize that $D$ may not be $\Rr$-Cartier.
\end{defn}


\subsection{Bounded families}
\begin{defn}\label{defn: bdd}
	A \emph{couple} $(X,D)$ consists of a normal projective variety $X$ and a reduced divisor $D$ on $X$. Two couples $(X,D)$ and $(X',D')$ are \emph{isomorphic} if there exists an isomorphism $X\rightarrow X'$ mapping $D$ onto $D'$.
	
	A set $\mathcal{P}$ of couples is \emph{bounded} if there exist finitely many projective morphisms $V^i\rightarrow T^i$ of varieties and reduced divisors $C^i$ on $V^i$ satisfying the following. For each $(X,D)\in\mathcal{P}$, there exist an index $i$ and a closed point $t\in T^i$, such that two couples $(X,D)$ and $(V^i_t,C^i_t)$ are isomorphic, where $V^i_t$ and $C^i_t$ are the fibers over $t$ of the morphisms $V^i\rightarrow T^i$ and $C^i\rightarrow T^i$ respectively. 
	
	A set $\mathcal{C}$ of projective pairs $(X,B)$ is said to be \emph{log bounded} if the set of the corresponding set of couples $\{(X,\Supp B)\}$ is bounded. A set $\mathcal{D}$ of projective varieties $X$ is said to be \emph{bounded} if the corresponding set of couples $\{(X,0)\}$ is bounded. A log bounded (resp. bounded) set is also called a \emph{log bounded family} (resp. \emph{bounded family}).
	
\end{defn}


\begin{thm}[{\cite[Corollary 1.2]{Bir16}}]\label{thm: BAB}
	Let $d$ be a positive integer and $\epsilon$ a positive real number. Then the projective varieties $X$ such that 
	\begin{enumerate}
		\item $(X,B)$ is projective $\epsilon$-lc of dimension $d$ for some boundary $B$,  and
		\item $K_X+B\sim_{\mathbb R}0$ and $B$ is big
	\end{enumerate}
	form a bounded family.
\end{thm}

\begin{rem}
	Assume that $X$ is of Fano type. Then we can run any $D$-MMP which terminates with some model $Y$ for any $\Rr$-Cartier $\Rr$-divisor $D$ on $X$ (cf. \cite[Corollary 2.9]{PS09}). Moreover, if $(X,\Delta)$ is $\epsilon$-lc for some positive real number $\epsilon$ and some boundary $\Delta$ such that $K_X+\Delta\sim_{\mathbb R}0$, then the $D$-MMP is a sequence of $(K_X+\Delta)$-flops. Therefore, let $\Delta_Y$ be the strict transform of $\Delta$ on $Y$, then $(Y,\Delta_Y)$ is $\epsilon$-lc and $K_Y+\Delta_Y\sim_{\mathbb R}0$. In particular, $Y$ belongs to a bounded family. We will repeatedly use this fact in the rest of the paper.
\end{rem}

The next corollary is well-known to experts:
\begin{cor}\label{cor: bab log bdd}
	Let $d$ be a positive integer and $\epsilon,\delta$ two positive real numbers. Then the projective pairs $(X,B)$ such that 
	\begin{enumerate}
	 \item $(X,\Delta)$ is projective $\epsilon$-lc of dimension $d$ for some boundary $\Delta$ such that $K_X+\Delta\sim_{\mathbb R}0$ and $\Delta$ is big,
		\item the non-zero coefficients of $B$ are $\geq\delta$, and
		\item $-(K_X+B)$ is pseudo-effective,
	\end{enumerate}
	form a log bounded family.
\end{cor}
\begin{proof}
By Theorem \ref{thm: BAB}, $X$ belongs to a bounded family, so there exist a positive integer $r$ depending only on $d$ and $\epsilon$, and a very ample divisor $H$ on $X$, such that $H^d\leq r$ and $(-K_X)\cdot H^{d-1}\leq r$. 
Since $-(K_X+B)$ is pseudo-effective, $-(K_X+B)\cdot H^{d-1}\geq 0$. Thus for any irreducible component $D$ of $B$, $D\cdot H^{d-1}\leq\frac{r}{\delta}$ is bounded from above. By \cite[Lemma 2.20]{Bir19}, $(X,B)$ is log bounded.
\end{proof}

\begin{lem}\label{lem: psd ft change to eff}
Let $d$ be a positive integer and $\epsilon,\delta$ two positive real numbers. Then there exists an integer $n$ depending only on $d$, $\epsilon$ and $\delta$ satisfying the following. Assume that 
    \begin{enumerate}
        \item $(X,\Delta)$ is projective $\epsilon$-lc of dimension $d$ for some boundary $\Delta$ such that $K_X+\Delta\sim_{\mathbb R}0$ and $\Delta$ is big, and
        \item $B'$ is a pseudo-effective Weil divisor on $X$ such that $-(\delta K_X+B')$ is pseudo-effective,
    \end{enumerate} 
    then there exists a $\Qq$-divisor $B''\ge0$ on $X$, such that $nB''\sim nB'$.
\end{lem}
\begin{proof}
Possibly taking a small $\Qq$-factorialization we may assume that $X$ is $\Qq$-factorial. Since $X$ is of Fano type, we may run a $B'$-MMP and reach a minimal model of $B'$, $f:X\dashrightarrow Y$. By Theorem \ref{thm: BAB}, $Y$ belongs to a bounded family. Thus there exist a positive integer $r$ depending only on $d$ and $\epsilon$, and a very ample divisor $H$ on $Y$, such that $H^d\leq r$ and $-K_Y\cdot H^{d-1}\leq r$. Let $B'_Y$ be the strict transform of $B'$ on $Y$. Since $-(\delta K_Y+B'_Y)$ is pseudo-effective, $-(\delta K_Y+B'_Y)\cdot H^{d-1}\ge0$, hence $B'_Y\cdot H^{d-1}\leq -\delta K_Y\cdot H^{d-1}\leq \delta r$. By \cite[Lemma 2.25]{Bir19}, there exists a positive integer $n_1$ depending only on $d,\epsilon$ and $\delta$, such that $n_1B'_Y$ is Cartier. 

We claim that $n:=2n_1(d+2)!(d+1)$ has the required properties. By \cite[Theorem 1.1]{Kol93}, $nB'_Y$ is base-point-free. In particular, there exists a $\Qq$-divisor $B''_Y\ge 0$, such that $nB''_Y\sim nB'_Y$. Let $p: W\rightarrow X$ and $q: W\rightarrow Y$ be a common resolution of $f:X\dashrightarrow Y$, then we have
$$p^*B'=q^*B'_Y+E$$
for some $\Qq$-divisor $E\geq 0$. Let
$$B'':=p_*q^*B''_Y+p_*E.$$
It follows that $nB''=n(p_*q^*B'_Y+p_*E)\sim np_{*}p^{*}B'=nB'$.
\end{proof}

Corollary \ref{cor: psd complement for bounded family} could be regarded as a generalization of the theory of complements to the setting of ``pseudo-effective boundaries" $B$ while $X$ lies in a bounded family of Fano type varieties. 

\begin{cor}\label{cor: psd complement for bounded family} 
 Let $d$ be a positive integer, $\epsilon$ a positive real number, and $\Ii_0$ a finite set of rational numbers. Then there exists a positive integer $n$ depending only on $d$, $\epsilon$ and $\Ii_0$ satisfying the following. Assume that
    \begin{enumerate}
        \item $(X,\Delta)$ is projective $\epsilon$-lc of dimension $d$ for some boundary $\Delta$ such that $K_X+\Delta\sim_{\mathbb R}0$ and $\Delta$ is big,
        \item $B=\sum_{i=1}^s b_iB_i$ is a pseudo-effective $\Rr$-Cartier $\mathbb R$-divisor on $X$, where each $b_i\in\Ii_0$ and each $B_i$ is a pseudo-effective Weil divisor, and
 \item $-(K_X+B)$ is pseudo-effective,
\end{enumerate}
    then there exists a $\Qq$-divisor $B^{+}$ on $X$, such that $nB^{+}$ is a Weil divisor, and $n(B^{+}-B)\in |-n(K_X+B)|$. In particular, $n(K_X+B^{+})\sim 0$.
\end{cor}

\begin{proof}
Let $n_0$ be a positive integer such that $n_0\Ii_0\subset\Zz$, and $B':=-n_0(K_X+B)$. Then both $B'$ and $-(n_0K_X+B')=n_0B$ are pseudo-effective. By Lemma \ref{lem: psd ft change to eff}, there exist a positive integer $n_1$ depending only on $d,\epsilon$ and $n_0$, and a $\Qq$-divisor $B''\ge 0$, such that $n_1B''\sim n_1B'=-n_0n_1(K_X+B)$. It follows that $n:=n_0n_1$, and $B^{+}:=\frac{1}{n}(n_1B''+nB)$ have the required properties.
\end{proof}

\subsection{Potentially birational}
\begin{defn}
Let $X$ be a normal projective variety and $D$ a big $\Qq$-Cartier $\Qq$-divisor on $X$. We say that $D$ is \emph{potentially birational} if for any two general closed points $x$ and $y$ on $X$, possibly switching $x$ and $y$, we may find $0\le \Delta\sim_{\mathbb Q}(1-\epsilon)D$ for some $\Qq$-divisor $\Delta$ and rational number $\epsilon\in (0,1)$, such that $(X,\Delta)$ is lc at $x$ with $\{x\}$ an lc center, and $(X,\Delta)$ is not klt at $y$.
\end{defn}

\begin{rem}
By definition, potentially birationality is preserved under $\Qq$-linear equivalence: if $D$ is potentially birational and $D\sim_{\mathbb Q}D'$, then $D'$ is also potentially birational. We will repeatedly use this fact in this paper.
\end{rem}

\begin{lem}[{\cite[Lemma 2.3.4]{HMX13}}]\label{lem: potentially birational}
Let $X$ be a normal projective variety and $D$ a big $\Qq$-Cartier $\Qq$-divisor on $X$.
\begin{enumerate}
    \item If $D$ is potentially birational, then $|K_X+\lceil D\rceil|$ defines a birational map.
    \item If $|D|$ defines a birational map, then $(2n+1)\lfloor D\rfloor$ is potentially birational.
\end{enumerate}
\end{lem}

\begin{lem}\label{lem: pot bir plus eff is pot bir}
Let $X$ be a normal projective variety, $D$ a big $\Qq$-Cartier $\Qq$-divisor on $X$, and $E$ an effective $\Qq$-Cartier $\Qq$-divisor on $X$. Suppose that $D$ is potentially birational. Then $D+E$ is potentially birational. 
\end{lem}
\begin{proof}
Since $E$ is effective, $E\sim_{\mathbb Q}F$ for some $\Qq$-divisor $F\geq 0$. Since $D$ is potentially birational, for any two general closed points $x$ and $y$ on $X$, possibly switching $x$ and $y$, we may find $0\le \Delta\sim_{\mathbb Q}(1-\epsilon)D$ for some rational number $\epsilon\in (0,1)$, such that $(X,\Delta)$ is lc at $x$ with $\{x\}$ an lc center, and $(X,\Delta)$ is not klt at $y$. Since $x,y$ are general, $x,y\not\in\Supp F$, hence $(X,\Delta':=\Delta+(1-\epsilon)F)$ is lc at $x$ with $\{x\}$ an lc center, and $(X,\Delta')$ is not klt at $y$. We have
$$(1-\epsilon)(D+E)\sim_{\mathbb Q}\Delta+(1-\epsilon)F\geq 0.$$
It follows that $D+E$ is potentially birational.
\end{proof}

The next technical result is crucial to the proofs of our main theorems:
\begin{thm}\label{thm: birational one m to inf m}
    Let $d,m_1,m_2$ be three positive integers and $\epsilon$ a positive real number. Then there exists a positive integer $m_0$ depending only on $d,m_1,m_2$ and $\epsilon$ satisfying the following. Assume that
    \begin{enumerate}
        \item $X$ is a $\Qq$-factorial projective variety of dimension $d$,
         \item $B=\sum_{i=1}^s b_iB_i$ is a pseudo-effective $\Rr$-Cartier $\mathbb R$-divisor on $X$, where each $B_i$ is a pseudo-effective Weil divisor,
         \item $D=n(K_X+B)$ is a big $\Rr$-divisor on $X$ for some integer $n$,
        \item $|m_2m_1nK_X+m_2\sum_{i=1}^s\lfloor m_1nb_i\rfloor B_i|$ defines a birational map, and
        \item for any $\Rr$-divisor $B'=\sum_{i=1}^sb_i'B_i$ on $X$, if $\max_{1\leq i\leq s}\{|b_i-b_i'|\}\leq\epsilon$, then $n(K_X+B')$ is big.
    \end{enumerate}
    Then $|mnK_X+\sum_{i=1}^s\lfloor mnb_i\rfloor B_i|$ defines a birational map for every positive integer $m\geq m_0$.
\end{thm}

\begin{proof}
Since $|m_2m_1nK_X+m_2\sum_{i=1}^s\lfloor m_1nb_i\rfloor B_i|$ defines a birational map, by Lemma \ref{lem: potentially birational}(2), $(2d+1)(m_2m_1nK_X+m_2\sum_{i=1}^s\lfloor m_1nb_i\rfloor B_i)$ is potentially birational. Let $m_3:=(2d+1)m_1m_2$. We will show that $m_0:=m_3+\lceil\frac{1}{\epsilon}\rceil+1$ satisfies our requirements. For any integer $m\geq m_0$, 
\begin{align*}
&(mnK_X+\sum_{i=1}^s\lfloor mnb_i\rfloor B_i-K_X)-(2d+1)(m_2m_1nK_X+m_2\sum_{i=1}^s\lfloor m_1nb_i\rfloor B_i)\\
=&(n(m-m_3)-1)(K_X+\sum_{i=1}^s(b_i-\frac{\{mnb_i\}}{n(m-m_3)-1})B_i)\\
&+(B+\frac{m_3}{m_1}\sum_{i=1}^s\{m_1nb_i\}B_i).\\
\end{align*}
Since 
$$|\frac{\{mnb_i\}}{n(m-m_3)-1}|<\frac{1}{|n(m-m_3)-1|}\leq\epsilon$$
for any $i$, by our assumptions,
$$n(K_X+\sum_{i=1}^s(b_i-\frac{\{mnb_i\}}{n(m-m_2)-1})B_i)$$ 
is big, hence 
$$(mnK_X+\sum_{i=1}^s\lfloor mnb_i\rfloor B_i-K_X)-(2d+1)(m_2m_1nK_X+m_2\sum_{i=1}^s\lfloor m_1nb_i\rfloor B_i)$$
is big. By Lemma \ref{lem: pot bir plus eff is pot bir}, $$mnK_X+\sum_{i=1}^s\lfloor mnb_i\rfloor B_i-K_X$$ is potentially birational. By Lemma \ref{lem: potentially birational}(1), $|mnK_X+\sum_{i=1}^s\lfloor mnb_i\rfloor B_i|$ defines a birational map. 
\end{proof}

\begin{rem}
In Theorem \ref{thm: birational one m to inf m}, $n$ can be either positive or negative ($n\not=0$ as $D=n(K_X+B)$ is big) and may depend on $X$. On the other hand, $m_0$ does not depend on $n$. In this paper, we usually take $n=1$ or $-1$.
\end{rem}

\section{Volume for bounded log Fano sub-pairs}

In this section, we study several properties of the volume $\vol(-(K_X+B))$, where $X$ is a Fano variety which belongs to a bounded family, and $B$ is a pseudo-effective $\Rr$-divisor. We prove Theorem \ref{thm: bddvolacc maynotbelc} at the end of this section.

\begin{lem}[{\cite[Lemma 2.5]{Jia18}, \cite[Lemma 4.2]{DS16}, \cite[Lemma 2.1]{CDHJS18}}]\label{lem: volineq}
	Let $X$ be a normal projective variety, $D$ an $\Rr$-Cartier
	$\Rr$-divisor, and $S$ a base-point free normal Cartier prime divisor. 
	Then for any positive real number $t$,
	\[
	\vol(D+tS)\leq \vol(D)+t(\dim X) \vol(D|_S+tS|_S).
	\]
\end{lem}

\begin{thm}\label{thm: approximation big bounded}
    Let $v$ be a positive real number and $\mathcal{P}$ a log bounded set of pairs. Then there exists a positive real number $\epsilon$ depending only on $v$ and $\mathcal{P}$ satisfying the following. Assume that
    \begin{enumerate}
        \item $(X,B)\in\mathcal{P}$ is a $\Qq$-factorial projective pair,
        \item $M$ is a pseudo-effective $\Rr$-divisor on $X$,
        \item $\vol(-(K_X+M))\geq v$, and
        \item $B'$ is an $\Rr$-divisor on $X$ such that $||B'||\leq\epsilon$ and $\Supp B'\subset\Supp B$.
    \end{enumerate}
    Then $\vol(-(K_X+B'+M))\geq\frac{v}{2}$. In particular, $-(K_X+B'+M)$ is big.
\end{thm}

\begin{proof}
Since $\mathcal{P}$ is log bounded, we may pick a very ample divisor $S$ on $X$, such that $S^d\leq r$, and $S+K_X$ and $S-\Supp B$ are big. Let $S'\in |S|$ be a general element, such that $S'$ is a normal Cartier prime divisor, and $(S+K_X+M)|_{S'}$ is effective.

By Lemma \ref{lem: volineq}, for any positive real number $t$,
\begin{align*}
   v&\leq\vol(-(K_X+M))\\
   &\leq\vol(-(K_X+M)-t S')+t(\dim X)\vol(-(K_X+M)|_{S'})\\
    &\leq\vol(-(K_X+M)-t S')+t(\dim X)\vol(S|_{S'})\\
    &=\vol(-(K_X+M)-t S')+t(\dim X)r.
\end{align*}
Since $\mathcal{P}$ is log bounded, we may assume that $\dim X\leq d$ for some positive integer $d$ depending only on $\mathcal{P}$. Let $\epsilon:=\frac{v}{2dr}$. For any $\Rr$-divisor $B'$ on $X$ such that $||B'||\leq\epsilon$ and $\Supp B'\subset\Supp B$,  we have
\begin{align*}
   \vol(-(K_X+B'+M))&\geq\vol(-(K_X+\epsilon\Supp B+M))\\
   &\geq\vol(-(K_X+M)-\epsilon S')\\
   &\geq v-dr\epsilon=\frac{v}{2}.
\end{align*}
Thus $\epsilon$ satisfies our requirements.
\end{proof}


The following Proposition is a special case of Theorem \ref{thm: bddvolacc maynotbelc}, i.e., when $\Ii$ is a finite set. The Step 2 in the proof of Proposition \ref{prop: bddvolfin maybenotlc} is similar to the proof of \cite[Lemma 3.26]{HLS19}.

\begin{lem}\label{prop: bddvolfin maybenotlc}
	Let $d$ be a positive integer, $\epsilon$ a positive real number, and $\Ii_0$ a finite set of non-negative numbers. Then there exists a finite set $\Ii_0'$ depending only on $d,\epsilon$ and $\Ii_0$ satisfying with the following. Assume that
	\begin{enumerate}
	    \item $(X,\Delta)$ is projective $\epsilon$-lc of dimension $d$ for some boundary $\Delta$ such that $K_X+\Delta\sim_{\mathbb R}0$ and $\Delta$ is big,
	    \item $B$ is a pseudo-effective $\Rr$-divisor on $X$, and
	    \item $B=\sum_{i=1}^s b_iB_i$, where each $b_i\in\Ii_0$, and each $B_i$ is a pseudo-effective Weil divisor.
	\end{enumerate}
then $\vol(-(K_X+B))\in \Ii_0'$.
\end{lem}

\begin{proof}
\noindent\textbf{Step 1}. In this step, we reduce the lemma to the case when each $B_i\geq 0$ and $X$ is $\Qq$-factorial. Possibly replacing $X$ with a small $\Qq$-factorization, we may assume that $X$ is $\Qq$-factorial. We may assume that $-(K_X+B)$ is big, otherwise $\vol(-(K_X+B))=0$ and there is nothing to prove. We may also assume that each $B_i\not\equiv 0$ and each $b_i>0$, otherwise we may replace $B$ with $B-b_iB_i$. Since $-(K_X+b_iB_i)- (-(K_X+B))$ is pseudo-effective, $-(\frac{1}{b_i}K_X+B_i)$ is big for any $i$. By Lemma \ref{lem: psd ft change to eff}, there exist a positive integer $n$ depending only on $d,\epsilon$ and $\Ii_0$, and $\Qq$-divisors $B_i'\geq 0$, such that $nB_i'\sim nB_i$ for each $i$. Possibly replacing $\Ii_0$ with $\frac{1}{n}\Ii_0$ and $B_i$ with $B_i'$, we may assume that each $B_i\geq 0$. 

By Theorem \ref{thm: BAB}, $X$ is bounded. Thus there exist a positive integer $r$ depending only on $d$ and $\epsilon$, and a very ample divisor $H$ on $X$, such that $(-K_X)\cdot H^{d-1}\leq r$. For any irreducible component $D$ of $B$, since $-(K_X+B)$ is big, $-(K_X+(\mult_DB)D)$ is big. Thus $-(K_X+(\mult_DB)D)\cdot H^{d-1}>0$, hence $\mult_DB<r$, and $B\in [0,r)$. Let 
$$\Ii_0':=\{\sum_{j=1}^sn_i\gamma_i\mid s,n_i\in\mathbb N, \gamma_i\in\Ii_0,\sum_{j=1}^sn_i\gamma_i<r\},$$
then $\Ii_0'$ is a finite set, and all the coefficients of $B$ belong to $\Ii_0'$. Possibly replacing $B_1,\dots,B_s$ with the irreducible components of $B$, and $\Ii_0$ with $\Ii_0'$, we may assume that $B_i$ are distinct prime divisors.
\medskip

\noindent\textbf{Step 2}. In this step, we prove the lemma by contradiction. Suppose that there exists a sequence of $\Qq$-factorial pairs $(X_i,B^i)$, such that $(X_i,\Delta_i)$ is $\epsilon$-lc and $K_{X_i}+\Delta_i\sim_{\mathbb R}0$ for some $\Delta_i$, $B^i\in\Ii_0$, $-(K_{X_i}+B^i)$ is big, and $\vol(-(K_{X_i}+B^i))$ is either strictly increasing or strictly decreasing. 
	
Since $X_i$ is of Fano type, possibly replacing $-(K_{X_i}+B^i)$ with a minimal model, we may assume that $-(K_{X_i}+B^i)$ is big and nef. By Corollary \ref{cor: bab log bdd}, $(X_i,B^i)$ belongs to a log bounded family.

	Possibly passing to a subsequence of $(X_i,\Supp B^i)$, we may assume that there exist a projective morphism $V\to T$ of varieties, a non-negative integer $u$, a reduced divisor $C=\sum_{j=1}^u C_j$ on $V$, and a dense set of closed points $t_i\in T$ such that $X_i$ is the fiber of $V\to T$ over $t_i$, and each component of $\Supp B^i$ is a fiber of $C_j\to T$ over $t_i$ for some $j$. Since $X_i$ is normal, possibly replacing $V$ with its normalization and replacing $C$ with its inverse image with reduced structure, we may assume that $V$ is normal. Possibly shrinking $T$, using Noetherian induction, and passing to a subsequence of $(X_i,\Supp B^i)$, we may assume that $V\to T$ is flat, and $C_j\to T$ is flat for any $j$.
	
	We may assume that $B^i=\sum_{j=1}^ub^i_j C_j|_{t_i}$, where $b^i_j\in\Ii_0$. Possibly passing to a subsequence, we may assume that $b^i_j$ is a constant for each $j$.
	
	For each $j$, let $a^i_j$ be a sequence of increasing rational numbers, such that $a^i_j\le b^1_j$ and $\lim_{i\to+\infty}a^i_j=b^1_j$. Let $B^1_i:=\sum_{j=1}^u a^i_j C_j|_{t_1}$ and $B^2_i:=\sum_{j=1}^u a^i_j C_j|_{t_2}$. By the asymptotic Riemann-Roch theorem and the invariance of Euler characteristic in a flat family, we have
	\begin{align*}
	&\vol(-(K_{X_1}+B^1))=(-(K_{X_1}+B^1))^d=\lim_{i\to+\infty}(-(K_{X_1}+B^1_i))^d\\ =&\lim_{i\to+\infty}(-(K_{X_2}+B^2_i))^d
	=(-(K_{X_2}+B^2))^d=\vol(-(K_{X_2}+B^2)),               
	\end{align*}
	a contradiction.
\end{proof}

Now we are ready to prove Theorem \ref{thm: bddvolacc maynotbelc}. The Step 2 in the proof of Theorem \ref{thm: bddvolacc maynotbelc} is similar to the proof of \cite[Proposition 3.28]{HLS19}.

\begin{proof}[Proof of Theorem \ref{thm: bddvolacc maynotbelc}] By Proposition \ref{prop: bddvolfin maybenotlc}, we only need to prove the case when $\Ii$ is a DCC set and show that the set of $\vol(-(K_X+B))$ satisfies the ACC.
\medskip

\noindent\textbf{Step 1}. In this step, we reduce the lemma to the case when each $B_i\geq 0$ and $X$ is $\Qq$-factorial. Possibly replacing $X$ with a small $\Qq$-factorization, we may assume that $X$ is $\Qq$-factorial. We may assume that $-(K_X+B)$ is big, otherwise $\vol(-(K_X+B))=0$ and there is nothing to prove. We may also assume that each $B_i\not\equiv 0$ and each $b_i>0$, otherwise we may replace $B$ with $B-b_iB_i$. Since $-(K_X+b_iB_i)- (-(K_X+B))$ is pseudo-effective, $-(\frac{1}{b_i}K_X+B_i)$ is big for any $i$. By Lemma \ref{lem: psd ft change to eff}, there exist a positive integer $n$ depending only on $d,\epsilon$ and $\Ii$, and $\Qq$-divisors $B_i'\geq 0$, such that $nB_i'\sim nB_i$ for each $i$. Possibly replacing $\Ii$ with $\frac{1}{n}\Ii$ and $B_i$ with $B_i'$, we may assume that each $B_i\geq 0$.

By Theorem \ref{thm: BAB}, $X$ is bounded. Thus there exist a positive integer $r$ depending only on $d$ and $\epsilon$, and a very ample divisor $H$ on $X$, such that $(-K_X)\cdot H^{d-1}\leq r$. For any irreducible component $D$ of $B$, since $-(K_X+B)$ is big, $-(K_X+(\mult_DB)D)$ is big. Thus $-(K_X+(\mult_DB)D)\cdot H^{d-1}>0$, hence $\mult_DB<r$, and $B\in [0,r)$. Let 
$$\Ii':=\{\sum_{j=1}^sn_i\gamma_i\mid s,n_i\in\mathbb N, \gamma_i\in\Ii,\sum_{j=1}^sn_i\gamma_i<r\},$$
then $\Ii'$ is a DCC set, and all the coefficients of $B$ belong to $\Ii'$. Possibly replacing $B_1,\dots,B_s$ with the irreducible components of $B$, and $\Ii$ with $\Ii'$, we may assume that $B_i$ are distinct prime divisors.

\medskip

\noindent\textbf{Step 2}. In this step, we prove the lemma by contradiction.

		Suppose that there exists a sequence of $\Qq$-factorial pairs $(X_i,B^i)$, such that $(X_i,\Delta_i)$ is $\epsilon$-lc and $K_{X_i}+\Delta_i\sim_{\mathbb R}-$ for some $\Delta_i$, $B^i\in\Ii$, $-(K_{X_i}+B^i)$ is big, and $\vol(-(K_{X_i}+B^i))$ is strictly increasing. 
	
 Since $X_i$ is of Fano type, possibly replacing $-(K_{X_i}+B^i)$ with a minimal model, we may assume that $-(K_{X_i}+B^i)$ is big and nef. By Corollary \ref{cor: bab log bdd}, $(X_i,B^i)$ belongs to a log bounded family. 
	
	In particular, the number of irreducible components of $B^i$ is bounded from above. Possibly passing to a subsequence, we may assume that
	$$B^i=\sum_{j=1}^u b^i_jB^i_j,$$
    where $u$ is a non-negative integer, $B^i_j$ are the irreducible components of $B^i$, $b^i_j\in\Ii$, and $\{b^i_j\}_{i=1}^{\infty}$ is an increasing sequence for every $j\in\{1,2,\dots,u\}$. Let
	$$\overline{b_j}:=\lim_{i\to +\infty}b^i_j,\text{ and }\overline{B^i}:=\sum_{j=1}^u \overline{b_j}B^i_j.$$

	\begin{claim}\label{claim: effective acc volume maynotbelc} There exists a sequence of positive real numbers $\epsilon_i$, such that $\lim_{i\to+\infty}\epsilon_i=0$, and
		$$\epsilon_i(-(K_{X_i}+B^i))-(\overline{B^i}-B^i)$$
		is effective.
	\end{claim}
	
	Suppose that the claim is true. Possibly passing to a subsequence, we may assume that $\epsilon_i<1$ for every $i$. Then 
	$$-(K_{X_i}+\overline{B^i})-(1-\epsilon_i)(-(K_{X_i}+B^i))=\epsilon_i(-(K_{X_i}+B^i))-(\overline{B^i}-B^i)$$
	is effective
	and $-(K_{X_i}+\overline{B^i})$ is big. Thus
	\begin{align*}
	\vol(-(K_{X_i}+\overline{B^i}))
	\ge&(1-\epsilon_i)^d\vol(-(K_{X_i}+B^i))\\
	\ge&(1-\epsilon_i)^d\vol(-(K_{X_i}+\overline{B^i})).
	\end{align*}	 
By Proposition \ref{prop: bddvolfin maybenotlc}, $\vol(-(K_{X_i}+\overline{B^i}))$ belongs to a finite set. Possibly passing to a subsequence, we may assume that $\vol(-(K_{X_i}+\overline{B^i}))=C$ is a constant. Hence 
	$$
	\frac{C}{(1-\epsilon_i)^d}
	\ge\vol(-(K_{X_i}+B^i))\\
	\ge C.
	$$	
	Since $\lim_{i\rightarrow\infty}\epsilon_i=0$ and $\vol(-(K_{X_i}+B^i))$ is increasing, we deduce that $\vol(-(K_{X_i}+B^i))$ belongs to a finite set, a contradiction.
\end{proof}

\begin{proof}[Proof of Claim \ref{claim: effective acc volume maynotbelc}] 
	Let $A_i$ be a very ample divisor on $X_i$, such that 
	$$\delta_i A_i-(\overline{B^i}-B^i)$$
	is effective for some positive real number $\delta_i$, $\lim_{i\to +\infty}\delta_i=0$, and $(-(K_{X_i}+B^i))^{d-1}\cdot A_i\leq r$, where $r$ is a positive real number depending only on $d,\epsilon$ and $\Ii$. There exists a positive real number $b$, such that $$b<\frac{(-(K_{X_i}+B^i))^{d}}{d(-(K_{X_i}+B^i))^{d-1}\cdot A_i}=\frac{\vol(-(K_{X_i}+B^i))}{d(-(K_{X_i}+B^i))^{d-1}\cdot A_i}.$$
	
	By Lemma \ref{lem: volineq}, we have
	\begin{align*}
	&\vol(-(K_{X_i}+B^i)-bA_i)\\
	\ge&     \vol(-(K_{X_i}+B^i))-bd\vol(-(K_{X_i}+B^i)|_{A_i})\\
	>&\vol(-(K_{X_i}+B^i))-\frac{\vol(-(K_{X_i}+B^i))}{(-(K_{X_i}+B^i))^{d-1}\cdot A_i}\cdot ((-(K_{X_i}+B^i))^{d-1}\cdot A_i)\\
	=&0,
	\end{align*}
	which implies that 
	$$-(K_{X_i}+B^i)-bA_i$$
	is effective.
	
	Let $\epsilon_i:=\frac{\delta_i}{b}$, then
	$$\epsilon_i(-(K_{X_i}+B^i))-(\overline{B^i}-B^i)=(\epsilon_i(-(K_{X_i}+B^i))-b\epsilon_iA_i)+(\delta_iA_i-(\overline{B^i}-B^i))$$
	is effective, and the claim is proved.
\end{proof}

\section{Proofs of Theorem \ref{thm: eff bir fin coeff gpair} and Theorem \ref{thm: eff bir bound volume more section gpair}}

\begin{lem}\label{lem: finite rational coefficient eff bir}
Let $d,n$ be two positive integers and $\epsilon$ a positive real number. Then there exists a positive integer $m$ divisible by $n$ depending only on $d,n$ and $\epsilon$ satisfying the following. Assume that
\begin{enumerate}
 \item $(X,\Delta)$ is projective $\epsilon$-lc of dimension $d$ for some boundary $\Delta$ such that $K_X+\Delta\sim_{\mathbb R}0$ and $\Delta$ is big,
 \item $B$ is a  pseudo-effective $\Qq$-divisor on $X$ such that $K_X+B$ is $\Qq$-Cartier,
 \item $nB$ is a Weil divisor, and
 \item $-(K_X+B)$ is big,
\end{enumerate}
then $|-m(K_X+B)|$ defines a birational map.
\end{lem}
\begin{proof}

Possibly replacing $X$ with a small $\Qq$-factorization, we may assume that $-(K_X+B)$ is $\Qq$-Cartier. Possibly replacing $-(K_X+B)$ with its minimal model, we may assume that $-(K_X+B)$ is big and nef. By Theorem \ref{thm: BAB}, $X$ is bounded. Thus there exist a positive integer $r$ depending only on $d$ and $\epsilon$, and a very ample Cartier divisor $H$ on $X$, such that $H^d\leq r$ and $-K_X\cdot H^{d-1}\leq r$. Since $B$ is pseudo-effective, $-n(K_X+B)\cdot H^{d-1}\leq -nK_X\cdot H^{d-1}\leq nr$. By \cite[Lemma 2.25]{Bir19}, there exists a positive integer $m_1$ depending only on $d,n$ and $\epsilon$ such that $m_1(K_X+B)$ is Cartier. By \cite[Theorem 1.1]{Kol93}, $m:=2m_1(d+2)!(d+1)$ satisfies our requirements.
\end{proof}

\begin{proof}[Proof of Theorem \ref{thm: eff bir bound volume more section gpair}(2.b)]
By Lemma \ref{lem: psd ft change to eff}, possibly replacing $X$ with a small $\Qq$-factorialization and each $B_i$ with its pullback, we may assume that $X$ is $\Qq$-factorial. Moreover, if we assume that each $B_i\geq 0$, then all the coefficients of $B\geq\delta$. Possibly replacing $B_1,\dots,B_s$ with the irreducible components of $B$, we may assume that $B_i$ are distinct prime divisors.

By Lemma \ref{lem: psd ft change to eff}, there exists a positive integer $n$ depending only on $d$ and $\epsilon$, and $\Qq$-divisors $B_i'\geq 0$ for every $i\in\{1,2,\dots,s\}$, such that $nB_i'\sim nB_i$ for each $i$.

Possibly reodering indices, we may let $\bar B:=\sum_{i=1}^{s_1}b_iB_i'+\sum_{i=s_1+1}^sb_iB_i'$ for some $s_1\in\{0,1,\dots,s\}$, such that $B'_i=0$ if and only if $i\geq s_1+1$. Since all the non-zero coefficients of $\bar B$ are $\geq\frac{\delta}{n}$, by Corollary $\ref{cor: bab log bdd}$, $(X,\bar B)$ is log bounded. Thus there exists a positive integer $u$ depending only on $d,\epsilon$ and $\delta$, such that $s_1\leq u$. Moreover, since $-(K_X+\bar B)$ is big, there exists a positive integer $N$ depending only on $d,\epsilon$ and $\delta$, such that for any $i\in\{1,2,\dots,s\}$ and any irreducible component $D$ of $\bar B$, $\mult_DB_i'\leq N$. 

By Theorem \ref{thm: approximation big bounded}, there exists a positive integer $n_0$ depending only on $d,\epsilon,\delta$ and $v$, such that $\vol(-(K_X+\bar B'))\geq\frac{v}{2}$ for any $\Rr$-divisor $B'$ on $X$ such that $||\bar B-B'||\leq\frac{1}{n_0}$ and $\Supp B'\subset\Supp\bar B$. 

Let $n_1:=Nun_0$. Then for any real numbers $b'_1,\dots,b'_s$ such that  $\max\{|b_i-b'_i|\}\leq\frac{1}{n_1}$ and any irreducible component $D$ of $\bar B$, \begin{align*}
    |\mult_D(\bar B-\sum_{i=1}^sb_i'B_i')|&=|\mult_D(\sum_{i=1}^s(b_i-b_i')B_i')|\leq\frac{1}{n_1}\sum_{i=1}^s|\mult_DB_i'|\\
    &=\frac{1}{n_1}\sum_{i=1}^{s_1}|\mult_DB_i'|\leq\frac{Nu}{n_1}=\frac{1}{n_0},
\end{align*}
which implies that $$||\bar B-\sum_{i=1}^sb_i'B_i'||\leq\frac{1}{n_0},$$
hence
$$\vol(-(K_X+\sum_{i=1}^sb_i'B_i))=\vol(-(K_X+\sum_{i=1}^sb_i'B'_i))\geq\frac{v}{2}.$$
In particular,
$$\vol(-(K_X+\sum_{i=1}^s\frac{\lceil n_1b_i\rceil}{n_1}B_i))\geq\frac{v}{2}.$$ 
By Lemma \ref{lem: finite rational coefficient eff bir}, there exists a positive integer $m_1$ depending only on $d,\epsilon,\delta$ and $v$, such that $$|-m_1(K_X+\sum_{i=1}^s\frac{\lceil n_1b_i\rceil}{n_1}B_i)|$$ defines a birational map and $n_1\mid m_1$. This is equivalent to say that
$$|\frac{m_1}{n_1}\cdot(-n_1)K_X+\frac{m_1}{n_1}\sum_{i=1}^s(\lfloor -n_1b_i\rfloor B_i)|$$
defines a birational map. The theorem follows from Theorem \ref{thm: birational one m to inf m}.
\end{proof}

\begin{proof}[Proof of Theorem \ref{thm: eff bir bound volume more section gpair}(1) and Theorem \ref{thm: eff bir bound volume more section gpair}(2.a)]
Possibly replacing $X$ with a small $\Qq$-factorialization, we may assume that $X$ is $\Qq$-factorial. Since $-(K_X+B)$ is big, $-K_X$ is also big and $\vol(-K_X)\geq v$. By Theorem \ref{thm: eff bir bound volume more section gpair}(2.b), there exists a positive integer $m_1$ depending only on $d$ and $\epsilon$, such that $|-m_1K_X|$ defines a birational map. By Lemma \ref{lem: potentially birational}(2), $-(2d+1)m_1K_X$ is potentially birational. 

By Lemma \ref{lem: psd ft change to eff}, there exists a positive integer $n$ depending only on $d$ and $\epsilon$ and a $\Qq$-divisor $P\geq 0$ on $X$, such that $-nK_X\sim nP$. By Corollary \ref{cor: bab log bdd}, $(X,P)$ is log bounded. By Theorem \ref{thm: approximation big bounded}, there exists a positive integer $\epsilon'$ depending only on $d,\epsilon$ and $v$, such that for any positive real number $\delta\leq\epsilon'$, $\vol(-(K_X+B+\delta P))\geq\frac{v}{2}$.

Let $m_2:=(2d+1)m_1$. We show that $m_0:=\lceil\frac{m_2+1}{\epsilon'}\rceil$ satisfies our requirements. For every  $m\geq m_0$, we have
\begin{align*}
    &((-mK_X-\lfloor mB\rfloor)-K_X)-(-(2d+1)m_1K_X)\\
    =&-m(\frac{m-m_2-1}{m}K_X+B)+\{mB\}\\
    \sim_{\mathbb R}&-m(K_X+B+\frac{m_2+1}{m}P)+\{mB\},
\end{align*}
and under the assumption of Theorem \ref{thm: eff bir bound volume more section gpair}(2.a),
\begin{align*}
    &((-mK_X-\sum_{i=1}^s\lfloor mb_i\rfloor B_i)-K_X)-(-(2d+1)m_1K_X)\\
    =&-m(\frac{m-m_2-1}{m}K_X+B)+\sum_{i=1}^s\{mb_i\}B_i\\
    \sim_{\mathbb R}&-m(K_X+B+\frac{m_2+1}{m}P)+\sum_{i=1}^s\{mb_i\}B_i.
\end{align*}

Since $m\geq\lceil\frac{m_2+1}{\epsilon'}\rceil$, $-m(K_X+B+\frac{m_2+1}{m}P)$ is big, which implies that
$$(-mK_X-\lfloor mB\rfloor)-K_X-(-(2d+1)m_1K_X)$$
is big, and under the assumption of Theorem \ref{thm: eff bir bound volume more section gpair}(2.a),
$$(-mK_X-\sum_{i=1}^s\lfloor mb_i\rfloor B_i-K_X)-(-(2d+1)m_1K_X)$$
is big. By Lemma \ref{lem: pot bir plus eff is pot bir}, 
$$(-mK_X-\lfloor mB\rfloor)-K_X$$
is potentially birational, and under the assumption of Theorem \ref{thm: eff bir bound volume more section gpair}(2.a),
$$(-mK_X-\sum_{i=1}^s\lfloor mb_i\rfloor B_i)-K_X$$
is potentially birational. Theorem \ref{thm: eff bir bound volume more section gpair}(1) and Theorem \ref{thm: eff bir bound volume more section gpair}(2.a) follow from Lemma \ref{lem: potentially birational}(1).
\end{proof}

\begin{proof}[Proof of Theorem \ref{thm: eff bir fin coeff gpair}]
For any $X$ and $B$ as in the assumption, by Theorem \ref{thm: bddvolacc maynotbelc}, $\vol(-(K_X+B))$ is bounded from below by a positive real number $v$ depending only on $d,\epsilon$ and $\Ii_0$. The Theorem follows from Theorem \ref{thm: eff bir bound volume more section gpair}.
\end{proof}

\section{Log general type pairs and effective Iitaka fibration}

In this section, we gather several state-of-the-art results on effective birationality and effective Iitaka fibrations, and slightly generalize them. 

\begin{thm}[{{{\cite[Theorem 1.1]{HM06}},\cite[Theorem 1.3(3)]{HMX14}},\cite[Theorem 1.3]{BZ16}}]\label{thm: BZ16 eff bir}
Let $d,n$ be two positive integers and $\Ii\subset [0,1]$ a DCC set. Then there exists a positive integer $m_0$ depending only on $d,n$ and $\Ii$ satisfying the following. Assume that
\begin{enumerate}
    \item $(X,B)$ is an lc pair of dimension $d$,
    \item $B\in\Ii$, 
    \item $M$ is a nef $\Qq$-divisor on $X$ such that $nM$ is Cartier, and
    \item $K_X+B+M$ is big,
\end{enumerate}
then $|m(K_X+B+M)|$ defines a birational map for any positive integer $m$ divisible by $m_0$.
\end{thm}


By using methods in this paper, we could slightly improve Theorem \ref{thm: BZ16 eff bir}:

\begin{thm}\label{thm: IMPROVED all big theorems on effective birationality}
Let $d,n$ be two positive integers and $\Ii\subset [0,1]$ a DCC set. Then there exists a positive integer $m_0$ depending only on $d,n$ and $\Ii$ satisfying the following. Assume that
\begin{enumerate}
    \item $(X,B)$ is an lc pair of dimension $d$,
    \item $B\in\Ii$, 
    \item $M$ is a nef $\Qq$-divisor on $X$ such that $nM$ is Cartier, and
    \item $K_X+B+M$ is big,
\end{enumerate}
then $|m(K_X+B+M)|$ defines a birational map for any positive integer $m\geq m_0$ such that $mM$ is a Weil divisor. In particular, if $K_X+B$ is big, then $|m(K_X+B)|$ defines a birational map for any positive integer $m\geq m_0$.
\end{thm}

\begin{proof}[Proof of Theorem \ref{thm: IMPROVED all big theorems on effective birationality}]
 Possibly replacing $\Ii$ with $\Ii\cup\{1\}$, $(X,B)$ with a dlt modification, and $M$ with its pullback, we may assume that $X$ is $\Qq$-factorial. 
 
  Assume that $B=\sum_{i=1}^sb_iB_i$ where $B_i$ are the irreducible components of $B$. Let $n_0$ be the least positive integer such that $M_0:=n_0M$ is a Weil divisor. Since $nM$ is Cartier, $n_0\mid n$. By Theorem \ref{thm: BZ16 eff bir}, there exists a positive integer $m_0'$ depending only on $d,n$ and $\Ii$ such that $|m_0'n(K_X+B+M)|$ defines a birational map. This is equivalent to say that
 $$|m_2m_1K_X+\sum_{i=1}^s\lfloor m_2m_1b_i\rfloor B_i+\lfloor\frac{m_2m_1}{n_0}\rfloor M_0|$$
 defines a birational map, where $m_2:=1$ and $m_1:=m_0'n$.
 
By \cite[Theorem 8.1]{BZ16}, there exists a positive real number $\epsilon$ depending only on $d,n$ and $\Ii$, such that for any $\Rr$-divisor $B'$ on $X$ and real number $\delta$, if $\Supp B'\subset\Supp B$,  $||B-B'||\leq\epsilon$, and $|\frac{1}{n_0}-\delta|<\epsilon$, then $K_X+B'+\delta M_0$ is big. By Theorem \ref{thm: birational one m to inf m}, there exists a positive integer $m_0$ depending only on $d,n$ and $\Ii$, such that
  $$|mK_X+\sum_{i=1}^s\lfloor mb_i\rfloor B_i+\lfloor\frac{m}{n_0}\rfloor M_0|$$
  defines a birational map for any positive integer $m\geq m_0$. In particular, when $mM$ is a Weil divisor, $\lfloor\frac{m}{n_0}\rfloor M_0=\frac{m}{n_0}M_0=mM$, hence $|m(K_X+B+M)|$ defines a birational map.
  
  When $K_X+B$ is big and $M=0$, we have $M_0=0$, hence $|m(K_X+B)|$ defines a birational map.
 \end{proof}

Theorem \ref{thm: IMPROVED all big theorems on effective birationality} immediately implies the following theorem on effective Iitaka fibrations, which slightly improves \cite[Theorem 1.2]{BZ16}.

\begin{thm}\label{thm: effective iitaka sufficiently large}
    Let $d,b$ and $\beta$ be three positive integers, then there exists a positive integer $m_0=m_0(d,b,\beta)$ depending only on $d,b$ and $\beta$ satisfying the following. Assume that
    \begin{enumerate}
        \item $W$ is a smooth projective variety of dimension $d$ such that $\kappa(W)\geq 0$,
        \item $V\rightarrow W$ is a resolution of $W$,
        \item $f: V\rightarrow X$ is an Iitaka fibration of $K_W$, and
        \item  $$b:=\min\{u\in\mathbb N^+\mid |uK_F|\not=\emptyset\},$$
        where $F$ is a very general fiber of $V\rightarrow X$,
        \item $\tilde F$ is the smooth model of the $\mathbb Z/(b)$-cover of $F$ over the unique divisor in $|bK_F|$,
        \item $\beta:=\dim H^{\dim\tilde F}(\tilde F,\mathbb C)$, and
        \item $N:=\lcm\{n\in\mathbb N^+\mid \varphi(n)\leq\beta\}$, where $\varphi$ is the Euler function.
    \end{enumerate}
    Then $|mK_W|$ defines an Iitaka fibration for any integer $m\geq m_0$ such that $Nb\mid m$.
\end{thm}

\begin{proof}[Proof of Theorem \ref{thm: effective iitaka sufficiently large}]
We follow the same lines as the proof of \cite[Theorem 1.2]{BZ16}. 

Possibly replacing $W$ with $V$, we may assume that the Iitaka fibration is $f: W\rightarrow X$. We may assume that $\kappa(W)\geq 1$, otherwise there is nothing to prove. We let 
$$\Ii=\Ii(b,N):=\{\frac{bNu-v}{bNu}\mid u,v\in\mathbb N^+,v\leq bN\}$$
which is a DCC set in $[0,1)$.

By \cite[Theorem 1.2]{VZ09} and \cite{FO00}, possibly replacing $W$ and $X$ with sufficiently high resolutions, we may assume that $X$ is smooth, and there exist a boundary $B$ on $X$, and a
nef $\Qq$-divisor $M$ on $X$, such that
\begin{itemize}
\item $B$ is a simple normal crossing $\Qq$-divisor whose coefficients are contained in $\Ii$, in particular, $(X,B)$ is lc,
    \item $NbM$ is nef Cartier,
    \item $K_X+B+M$ is big, and 
    \item for any positive integer $m$ divisible by $b$, $|mK_W|$ defines an Iitaka fibration if and only if $|m(K_X+B+M)|$ defines a birational map.
\end{itemize}
Theorem \ref{thm: effective iitaka sufficiently large} immediately follows from Theorem \ref{thm: IMPROVED all big theorems on effective birationality}.
\end{proof}



\section{Examples}\label{sec: example for eff bir}

In this section we give examples which show that the assumptions of our main theorems are necessary. We state the following lemma, which provides a very convenient necessary condition for a divisor to define a birational map, and will be repeatedly used in this section.
\begin{lem}[c.f. {\cite[Lemma 2.2]{HM06}}]\label{lem: birational vol 1}
    Let $X$ be a normal projective variety and $D$ a $\Qq$-Cartier Weil divisor on $X$ such that $|D|$ defines a birational map. Then $\vol(D)\geq 1$.
\end{lem}

\subsection{Necessity of integrability of the nef part}
The next example shows that we need to study linear systems of the forms  $|-mK_X-\sum_{i=1}^s\lceil mb_i\rceil B_i|$  instead of $|-m(K_X+B)|$ in Theorem \ref{thm: eff bir fin coeff gpair} and Theorem \ref{thm: eff bir bound volume more section gpair}(2) when $B\not\geq 0$, and also shows that the requirement ``$mM$ is a Weil divisor" in Theorem \ref{thm: IMPROVED all big theorems on effective birationality} is necessary.

\begin{ex}
Let $n$ be a positive integer, and $X$ a curve of genus $0$ (resp. genus $2$). For every positive integer $s$, let $p_1,\dots,p_{2s}$ be $2s$ general closed points on $X$, and let $B^s=M_s:=\frac{1}{n}\sum_{i=1}^s(p_{2i-1}-p_{2i})$. Notice that $\deg(-(K_X+B^s))=2$ (resp. $\deg(K_X+M_s)=2$), $\sum_{i=1}^s(p_{2i-1}-p_{2i})$ is a nef Cartier divisor, and $mM_s$ is a Weil divisor if and only if $n\mid m$.

For any positive integer $m<\frac{s}{2}$ such that $n\nmid m$,
$$\deg(\lfloor -m(K_X+B^s)\rfloor)=2m-s<0\ (\text{resp.}\ \deg(\lfloor m(K_X+M_s)\rfloor)=2m-s<0).$$
Therefore, if $m<\frac{s}{2}$, then $|\lfloor -m(K_X+B^s)\rfloor|$ (resp. $|\lfloor m(K_X+M_s)\rfloor|$) defines a birational map if and only if $n\mid m$. 
\end{ex}

The next example shows that in Theorem \ref{thm: IMPROVED all big theorems on effective birationality}, $m_2$ relies on the Cartier index of $M$ (i.e. relies on $n$) even when $\dim X=2, B=0$, and $M$ is integral. 

\begin{ex}
For any positive integer $n$, let $X_n:=\mathbb P(1,1,n)$, $H_{1,n}$ the first toric invariant divisor (i.e. corresponds to the first $1$), and $M_n=-K_{X_n}+H_{1,n}$. Then $M_n$ is nef integral and $K_{X_n}+M_n=H_{1,n}$ is big. However, since $$\deg(m(K_{X_n}+M_n))=m\deg H_{1,n}=m,$$
$m(K_{X_n}+M_n)$ does not define a birational map when $m<n$.
\end{ex}

\subsection{Necessity of the lower bound of the volume}

 The following example shows that the assumption $\vol(-(K_X+B))\geq v$ in Theorem \ref{thm: eff bir bound volume more section gpair}(2.b) is necessary even when $\dim X=1$ and each $B_i\geq 0$:
 
 \begin{ex}\label{ex: p1 dcc vol no bound counter}
Let $X:=\mathbb P^1$, $p_1,p_2,p_3,p_4$ four distinct closed points on $X$, and $B^n:=(\frac{1}{2}-\frac{1}{2n})(p_1+p_2+p_3+p_4)$ for every integer $n\geq 2$. Then all the coefficients of $B^n$ are $\geq\frac{1}{4}$, $(X,B^n)$ is $\frac{1}{2}$-lc log Fano for every $n$, and $\deg(-(K_X+B^n))=\frac{2}{n}$. For every positive integer $m<n$, $\deg(\lfloor-m(K_X+B^n)\rfloor)\leq 0$, so $|-m(K_X+B^n)|$ does not define a birational map.
 \end{ex}
 
 \subsection{Necessity of the lower bound of the coefficients}
 
 The following example shows that the assumption ``$b_i\geq\delta$" in Theorem \ref{thm: eff bir bound volume more section gpair}(2.b) is necessary even when $\dim X=1$ and each $B_i\geq 0$:

\begin{ex}\label{ex: p1 with many points}
Consider $X:=\mathbb P^1$ and $B^n:=\frac{1}{2n}\sum_{i=1}^{2n}p_i$ for every positive integer $n$, where $p_i$ are distinct closed points on $X$. Then $(X,B^n)$ is $\frac{1}{2}$-lc and $\deg(-(K_X+B^n))=1$ for every $n$. However, $|-m(K_X+B^n)|$ does not define a birational map for any $m\leq 2n-1$.
\end{ex}

\subsection{Necessity of the singularity assumption}
The following example shows that the assumption ``$(X,\Delta)$ is projective $\epsilon$-lc of dimension $d$ for some boundary $\Delta$ such that $K_X+\Delta\sim_{\mathbb R}0$ and $\Delta$ is big" is necessary in Theorem \ref{thm: eff bir fin coeff gpair} and Theorem \ref{thm: eff bir bound volume more section gpair}, even when $\dim X=2$ and $B=0$. We remark that when $\dim X=1$, $X=\mathbb P^1$ and the assumption is automatically satisfied by letting $\epsilon=\frac{1}{2}$ and $\Delta=\frac{1}{2}\sum_{i=1}^4p_i$, where $p_i$ are different closed points on $X$.
 
 \begin{ex}\label{ex: gen hyp p111n}
Consider the general hypersurfaces $X_{n+1}\subset\mathbb P(1,1,1,n)$ of degree $n+1$ for every positive integer $n$. By \cite[Theorem 8.1]{IF00}, $X$ is quasismooth klt Fano, and $\vol(-K_{X_{n+1}})=\frac{4(n+1)}{n}>4$, but $X_{n+1}$ is not $\frac{2}{n}$-klt.

Notice that the map defined by $|-mK_{X_{n+1}}|$ cannot be birational for any positive integer $m<\frac{n}{2}$. Indeed, for any $n\geq 3$ and positive integer $m<\frac{n}{2}$, the map given by $|-mK_{X_{n+1}}|$ must factorize through $\mathbb P(1,1,1)\cong\mathbb P^2$. Since $\deg X_{n+1}=n+1>1+1+1$, the map defined by $|-mK_{X_{n+1}}|$ cannot be birational.
\end{ex}
 
The following well-known example shows that the assumption ``$(X,\Delta)$ is projective $\epsilon$-lc of dimension $d$ for some boundary $\Delta$ such that $K_X+\Delta\sim_{\mathbb R}0$ and $\Delta$ is big" is necessary in Theorem \ref{thm: bddvolacc maynotbelc} even when $\dim X=2$ and $B=0$. Notice that when $d=1$, $(X=\mathbb P^1,\Delta=\frac{1}{2}\sum_{i=1}^4p_i)$ is $\frac{1}{2}$-lc, where $p_i$ are different closed points on $X$.
\begin{ex}[{\cite[22.5 Proposition]{KM99}, \cite[Example 2.1.1]{HMX14}}]
Let $p,q,r$ be three coprime positive integers. Then $\mathbb P(p,q,r)$ is a klt del Pezzo surface such that $\vol(-K_{\mathbb P(p,q,r)})=\frac{(p+q+r)^2}{pqr}$. However  $$\{\frac{(p+q+r)^2}{pqr}\mid p,q,r\in\mathbb N^+,(p,q)=(q,r)=(p,r)=1\}$$
is dense in $\mathbb R^+$.
\end{ex}

\subsection{The linear system $|-\lfloor m(K_X+B)\rfloor|$}

Even if we replace $|-m(K_X+B)|$ with the linear system $|-\lfloor m(K_X+B)\rfloor|$ which has more sections, the assumption ``$B\in\Ii_0$" in Theorem \ref{thm: eff bir fin coeff gpair} and the assumption ``$\vol(-(K_X+B))\geq v$" in Theorem \ref{thm: eff bir bound volume more section gpair}(1) and Theorem \ref{thm: eff bir bound volume more section gpair}(2.a) are still necessary, even when $\dim X\geq 2$, the coefficients of $B$ belong a DCC set of non-negative real numbers, and $(X,B)$ is $\epsilon$-lc for some $\epsilon>0$. We remark that when $\dim X=1$, $\deg(-\lfloor m(K_X+B)\rfloor)\geq 1$ so $|-\lfloor m(K_X+B)\rfloor|$ defines a birational map for any positive integer $m$.
 
\begin{ex}\label{ex: gen hyp p11910}
Let 
$$X=X_{18}\subset W:=\mathbb P(2,3,5,9)$$
be a general hypersurface of degree $18$. By \cite[Theorem 8.1]{IF00}, $X$ is quasismooth. By \cite[Lemma 3.1.4; Part 4, The Big Table, page 135]{CPS10}, $X$ is an exceptional del Pezzo surface and $\vol(-K_X)=\frac{1}{15}$.

Let $H\subset W$ be the toric invariant divisor corresponds to the first coordinate, i.e. corresponds to the $2$ in $(2,3,5,9)$, and $B:=H|_X$. Since $K_W+X+\frac{1}{2}H\sim_{\Qq}0$, $K_X+\frac{1}{2}B\sim_{\Qq}0$. Since $X$ is exceptional, $(X,\frac{1}{2}B)$ is klt, and hence $\epsilon$-lc for some real number $\epsilon>0$. (Indeed, by \cite[
Lemma 3.1.4]{CPS10}, $(X,0)$ is $\frac{3}{5}$-lc. Since $(X,B)$ is lc, we may take $\epsilon=\frac{3}{10}$.)

Consider $(X,B_n:=(\frac{1}{2}-\frac{1}{2n})B)$.
Then $(X,B_n)$ is $\epsilon$-lc log Fano for every positive integer $n$. Moreover, for every positive integer $m<n$,
$$\lfloor m(K_X+B_n)\rfloor=mK_X+\lfloor\frac{m-1}{2}\rfloor B\sim_{\Qq} (m-2\lfloor\frac{m-1}{2}\rfloor)K_X.$$
Since $\vol(-K_X)=\frac{1}{15}$, 
$$\vol(-\lfloor m(K_X+B_n)\rfloor)=\vol(-(m-2\lfloor\frac{m-1}{2}\rfloor)K_X)\leq\vol(-2K_X)=\frac{4}{15}<1.$$ 

By Lemma \ref{lem: birational vol 1},  $|-\lfloor m(K_X+B_n)\rfloor|$ does not define a birational map for any positive integer $m<n$.
\end{ex}


\begin{thebibliography}{99}


		\bibitem[Bir16]{Bir16} C. Birkar, \textit{Singularities of linear systems and boundedness of Fano varieties}. arXiv: 1609.05543v1, 2016.
	
	\bibitem[Bir19]{Bir19} C. Birkar, \textit{Anti-pluricanonical systems on Fano varieties}. Ann. of Math. (2), 190(2):345--463, 2019.


\bibitem[BZ16]{BZ16} C. Birkar and D.-Q. Zhang, \textit{Effectivity of Iitaka fibrations and pluricanonical systems of polarized pairs}. Publ. Math. Inst. Hautes Etudes Sci., \textbf{123}, 283--331 (2016).




	\bibitem[CDHJS18]{CDHJS18} W. Chen, G. Di Cerbo, J. Han, C. Jiang and R. Svaldi, \textit{Birational boundedness of rationally connected
	Calabi-Yau 3-folds}. arXiv: 1804.09127.
	
\bibitem[Che20]{Che20} G. Chen, \textit{Boundedness of $n$-complements for generalized polarized pairs}, arXiv: 2003.04237v1, 2020.

\bibitem[CPS10]{CPS10} I. Cheltsov, J. Park, and C. Shramov, \textit{Exceptional del Pezzo hypersurfaces}, arXiv: 0810.2704v6. Short version published on Journal of Geometric Analysis, \textbf{20}, 787–-816, 2010.

\bibitem[DS16]{DS16} G. Di Cerbo, R. Svaldi, \textit{Birational boundedness of low dimensional elliptic Calabi-Yau varieties with a section}. arXiv:1608.02997.





\bibitem[FO00]{FO00} O. Fujino and S. Mori, \textit{A canonical bundle formula}, J. Differential Geom. \textbf{56} (2000), no. 1, 167-188.

\bibitem[HLS19]{HLS19} J. Han, J. Liu, and V. V. Shokurov, \textit{ACC for minimal logdiscrepancies ofexceptional singularities}, arXiv:1903.04338v2, 2019.

\bibitem[HL18]{HL18}
J. Han and Z. Li, \emph{Weak Zariski decomposition and log terminal models for generalized polarized pairs}. arXiv:1806.01234v1, 2018.

\bibitem[HL19]{HL19}
J. Han and W. Liu, \emph{On a generalized canonical bundle formula for generically finite morphisms}. arXiv:1905.12542, 2019.

\bibitem[HL20]{HL20}
J. Han and W. Liu, \emph{On numerical nonvanishing for generalized log canonical pairs}. Documenta Math., \textbf{25} (2020). 93–123.


\bibitem[HM06]{HM06} C. Hacon and J. M\textsuperscript{c}Kernan, \textit{Boundedness of pluricanonical maps of varieties of general type}, Invent. Math. \textbf{166} (2006), no. 1, 1–25.


\bibitem[HMX13]{HMX13} C.D. Hacon, J. M\textsuperscript{c}Kernan and C. Xu, \textit{On the birational automorphisms of varieties of general type}, Ann. of Math. (2) \textbf{177} (2013), no. 3, 1077-–1111.


	\bibitem[HMX14]{HMX14} C.D. Hacon, J. M\textsuperscript{c}Kernan, and C. Xu, \textit{ACC for log canonical thresholds}. Ann. of Math. (2) \textbf{180} (2014), no. 2, 523--571.
	
		

	\bibitem[IF00]{IF00} A.R. Iano-Fletcher, \textit{Working with weighted complete intersections}, In: Explicit birational geometry of $3$-folds, pp. 101–173, London Math. Soc. Lecture Note Ser., \textbf{281}, Cambridge Univ. Press, Cambridge, 2000.

	\bibitem[Jia18]{Jia18} C. Jiang, \textit{On birational boundedness of {F}ano fibrations}. Amer. J. Math., 140(5), 2018: 1253--1276.
	
	\bibitem[KM99]{KM99} S. Keel and J. M\textsuperscript{c}Kernan, \textit{Rational curves on quasi-projective surfaces}, Mem. Amer. Math. Soc. \textbf{140} (1999), no. 669, viii+153.
	
	
	
 \bibitem[Kol93]{Kol93} J. Koll\'{a}r, \textit{Effective base point freeness}. Math. Annalen (1993), Volume \textbf{296}, Issue 1, pp 595–-605




\bibitem[LT19]{LT19} V. Lazi\'{c} and N. Tsakanikas, \textit{On the existence of minimal models for log canonical pairs}, arXiv:1905.05576, to appear in Publ. Res. Inst. Math. Sci.


	\bibitem[PS09]{PS09} Y.G. Prokhorov and V.V. Shokurov, \textit{Towards the second main theorem on complements}. J. Algebraic Geom., \textbf{18}(1), 2009: 151--199.

\bibitem[Tak06]{Tak06} S. Takayama, \textit{Pluricanonical systems on algebraic varieties of general type}, Invent. Math. \textbf{165} (2006), no. 3, 551–587

\bibitem[VZ09]{VZ09} E. Viehweg and D.-Q. Zhang, \textit{Effective Iitaka fibrations}, J. Algebraic Geom., \textbf{18} (2009), no. 4, 711--730.

\end{thebibliography}
\end{document}